\documentclass[12pt,leqno]{article}
\usepackage{amsmath,amssymb,amsfonts,amsthm,latexsym,amstext,amscd}
\usepackage[english]{babel}
\usepackage{fancyhdr}

\newtheorem{theorem}{Theorem}

\newtheorem{lemma}{Lemma}

\theoremstyle{definition}

\newcommand{\beq}{\begin{equation}}
\newcommand{\eeq}{\end{equation}}

\begin{document}

\title
{Patterns of primes in arithmetic progressions}

\author
{J\'anos Pintz\thanks{Supported by OTKA Grants NK104183, K100291 and ERC-AdG.~321104.}}

\date{}

\numberwithin{equation}{section}


\maketitle

\section{Introduction}
\label{sec:1}

In their ground-breaking work Green and Tao \cite{GT2008} proved the existence of infinitely many $k$-term arithmetic progressions in the sequence of primes for every integer $k > 0$.
I showed a conditional strengthening of it \cite{Pin2010} according to which if the primes have a distribution level $\vartheta > 1/2$ (for the definition of the distribution level see \eqref{eq:1.1} below), then there exists a constant $C(\vartheta)$ such that we have a positive even $d \leqslant C(\vartheta)$ with the property that $0 < d \leqslant C(\vartheta)$ and for every $k$ there exist infinitely many arithmetic progressions $\{p_i^*\}_{i = 1}^k$ of length $k$ with $p_i^* \in \mathcal P$ ($\mathcal P$ denotes the set of primes) such that $p_i^* + d$ is a prime too, in particular, the prime following $p_i^*$.
After the proof of Zhang \cite{Zhang2014}, proving the unconditional existence of infinitely many bounded gaps between primes (this was proved earlier in our work \cite{GPY2009} under the condition that primes have a distribution level $\vartheta > 1/2$) I showed this without any unproved hypotheses \cite{Pin2015}.

We say that $\theta$ is a distribution level of the primes if
\beq
\label{eq:1.1}
\sum_{q \leqslant x^\theta} \max_{\substack{a\\ (a,q) = 1}} \left| \pi(x, q, a) - \frac{\pi(x)}{\varphi (q)} \right| \ll_{A} \frac{x}{(\log x)^A}
\eeq
holds for any $A > 0$ where the $\ll$ symbol of Vinogradov means that $f(x) = O(g(x))$ is abbreviated by $f(x) \ll g(x)$.

In his recent work James Maynard \cite{May2015} gave a simpler and more efficient proof of Zhang's theorem.
In particular he gave an unconditional proof of a weaker version of Dickson's conjecture \cite{Dic1904} which we abbreviate as Conjecture DHL since Hardy and Littlewood formulated a stronger quantitative version of it twenty years later \cite{HL1923}.

\smallskip
\noindent
{\bf Conjecture} DHL (Prime $k$-tuples Conjecture).
{\it Let $\mathcal H = \{h_1, \dots, h_k\}$ be admissible, which means that for every prime $p$ there exists an integer $a_p$ such that for any $i$ $a_p \not\equiv h_i (\text{\rm mod }p)$.
Then there are infinitely many integers $n$ such that all of $n + h_1,\dots, n + h_k$ are primes.}

\smallskip
The weaker version showed by Maynard (and simultaneously and independently by T. Tao (unpublished)) was that Conjecture DHL $(k, k_0)$ (formulated below) holds for $k \gg k_0^2 e^{4k_0}$.

\smallskip
\noindent
{\bf Conjecture DHL$(k, k_0)$}.
{\it If $\mathcal H$ is admissible of size $k$, then there are infinitely many integers $n$ such that
$\{n + h_i\}_{i = 1}^k$ contains at least $k_0$ primes.}

\smallskip
A brief argument, given by Maynard \cite{May2015} (see Theorem 1.2 of his work) shows that if there exists a $C(k_0)$ such that DHL$(k,k_0)$ holds for $k \geqslant C(k_0)$, then a positive proportion of all admissible $m$-tuples satisfy the prime $m$-tuple conjecture for every $m$ (for the exact formulation see Theorem 1.2 of \cite{May2015}).

The purpose of the present work is to show a common generalization of the result of Maynard (and Tao) and that of Green--Tao.

\begin{theorem}
\label{th:1}
Let $m > 0$ and $\mathcal A = \{a_1, \dots, a_n\}$ be a set of $r$ distinct integers with $r$ sufficiently large depending on $m$.
Let $N(\mathcal A)$ denote the number of integer $m$-tuples $\{h_1, \dots, h_n\} \subseteq \mathcal A$ such that there exist for every $\ell$ infinitely many $\ell$-term arithmetic progressions of primes $\{p_i^*\}_{i = 1}^\ell$ where $p_i^* + h_j$ is also prime for each pair $i, j$.
Then
\beq
\label{eq:1.2}
N(\mathcal A) \gg_m \# \bigl\{(h_1, \dots, h_m) \in \mathcal A \bigr\} \gg_m |A|^m = r^m.
\eeq
\end{theorem}

This is an unconditional generalization of the result in \cite{Pin2010}.

\section{Preparation. First part of the proof of Theorem~\ref{th:2}}
\label{sec:2}

The arguments in the last three paragraphs of Section~4 of \cite{May2015} can be applied here practically without any change and so, similarly to Theorems 1.1 and 1.2 of \cite{May2015}, our Theorem~\ref{th:1} will also follow in essentially the same way from (the weaker)

\begin{theorem}
\label{th:2}
Let $m$ be a positive integer, $\mathcal H = \{h_1, \dots, h_k\}$ be an admissible set of $k$ distinct non-negative integers $h_i \leqslant H$, $k = \lceil C m^2 e^{4m}\rceil$ with a sufficiently large absolute constant~$C$.
Then there exists an $m$-element subset
\beq
\label{eq:2.1}
\{h_1', h_2', \dots, h_m'\} \subseteq \mathcal H
\eeq
such that for every positive integer $\ell$ we have infinitely many $\ell$-element non-trivial arithmetic progressions of primes $p_i^*$ such that $p_i^* + h_j' \in \mathcal P$ for $1 \leqslant  i \leqslant \ell$, $1 \leqslant j \leqslant m$, further $p_i^* + h_j'$ is always the $j$-th prime following $p_i^*$.
\end{theorem}

\noindent
{\bf Remark}.

\begin{itemize}
\item[(i)] For $\ell = m = 1$ this is Zhang's theorem,

\item[(ii)] for $\ell = 1$, $m$ arbitrary this is the Maynard--Tao theorem,

\item[(iii)] for $m = 0$, $\ell$ arbitrary this is the Green--Tao theorem,

\item[(iv)] for $m = 1$, $\ell$ arbitrary this was proved under the condition that primes have a distribution level $\theta > 1/2$ in \cite{Pin2010}, unconditionally (using Zhang's method) in \cite{Pin2015}.
\end{itemize}

In order to show our Theorem~\ref{th:2} we will follow the scheme of \cite{May2015}.
We therefore emphasize just a few notations here, but we will use everywhere Maynard's notation throughout our work.
Similarly to his work, $k$ will be a fixed integer, $\mathcal H = \{h_1, \dots, h_k\} \subseteq [0, H]$ a fixed admissible set.
Any constants implied by the $\ll$ and $0$ notations may depend on $k$ and~$H$.
$N$~will denote a large integer and asymptotics will be understood as $N \to \infty$.
Most variables will be natural numbers, $p$ (with or without subscripts) will denote always primes, $[a, b]$ the least common multiple of $[a, b]$ (however, sometimes the closed interval $[a, b]$).
We will weight the integers with a non-negative weight $w_n$ which will be zero unless $n$ lies in a fixed residue class $\nu_0$ $(\text{\rm mod }W)$ where $W = \prod\limits_{p \leqslant D_0} p$.
$D_0$ tends in \cite{May2015} slowly to infinity with~$N$.
His choice is actually $D_0 = \log\log\log N$.
However, it is sufficient to choose
\beq
\label{eq:2.1masodszor}
D_0 = C^*(k),
\eeq
with a sufficiently large constant $C^*(k)$, depending on~$k$.

The proof runs similarly in this case as well just we lose the asymptotics then, but the dependence on $D_0$ is explicitly given in \cite{May2015}.
The weights $w_n$ are defined in (2.4) of \cite{May2015} as
\beq
\label{eq:2.2}
w_n = \biggl(\sum_{d_i \mid n + h_i\, \forall i} \lambda_{d_1, \dots, d_k}\biggr)^2.
\eeq

The choice of $\lambda_{d_1, \dots, d_k}$ will be through the choice of other parameters $y_{r_1, \dots, r_k}$ by the aid of the identity
\beq
\label{eq:2.3}
\lambda_{d_1, \dots, d_k} = \biggl(\prod\limits_{i = 1}^k \mu(d_i) d_i\biggr) \sum_{\substack{r_1, \dots, r_k\\
d_i \mid r_i \,\forall i\\
(r_i, W) = 1}} \frac{\mu\Bigl(\prod\limits_{i = 1}^k r_i\Bigr)^2}{\prod\limits_{i = 1}^k \varphi(r_i)} y_{r_1, \dots, r_k}
\eeq
whenever $\Bigl(\prod\limits_{i = 1}^k d_i, W\Bigr) = 1$ and $\lambda_{d_1, \dots, d_r} = 0$ otherwise.
Here $y_{r_1, \dots, r_k}$ will be defined by the aid of a piecewise differentiable function $F$, the distribution $\theta > 0$ of the primes, with $R = N^{\theta / 2 - \varepsilon}$ as
\beq
\label{eq:2.4}
y_{r_1, \dots, r_k} = F \left(\frac{\log r_1}{\log R}, \dots, \frac{\log r_k}{\log R}\right)
\eeq
where $F$ will be real valued, supported on
\beq
\label{eq:2.5}
R_k = \Biggl\{(x_1, \dots, x_k) \in [0, 1]^k : \sum_{i = 1}^k x_i \leqslant 1\Biggr\}.
\eeq

All this is in complete agreement with the notation of Proposition~1 and (6.3) of \cite{May2015}.

Our proof will also make use of the main pillars of Maynard's proof, his Propositions 1--3, which we quote now with the above notations as

\smallskip
\noindent
{\bf Proposition 1'}.
{\it With the above notation let
\beq
\label{eq:2.6}
S_1 := \sum_{\substack{n\\
N \leqslant n < 2N\\
n \equiv \nu_0 \ (\text{\rm mod }W)}} w_n, \ \ \
S_2 := \sum_{\substack{n\\
N \leqslant n < 2N\\
n \equiv \nu_0 \ (\text{\rm mod }W)}} \biggl( w_n  \sum_{i = 1}^k \chi_{\mathcal P}(n + h_i) \biggr) ,
\eeq
where $\chi_{\mathcal P}(n)$ denotes the characteristic function of the primes.
Then we have as $N \to \infty$
\beq
\label{eq:2.7}
S_1 = \frac{\left(1 + O\left(\frac1{D_0}\right)\right) \varphi(W)^k N(\log R)^k}{W^{k + 1}} I_k(F),
\eeq
\beq
\label{eq:2.8}
S_2 = \frac{\left(1 + O\left(\frac1{D_0}\right)\right) \varphi(W)^k N(\log R)^{k + 1}}{W^{k + 1}} \sum_{j = 1}^k J_k^{(j)}(F),
\eeq
provided $I_k(F) \neq 0$ and $J_k^{(j)}(F) \neq 0$ for each $j$, where}
\beq
\label{eq:2.9}
I_k(F) = \int\limits_0^1 \dots \int\limits_0^1 F(t_1, \dots, t_k)^2 dt_1 \dots dt_k,
\eeq
\beq
\label{eq:2.10}
I_k^{(j)}(F) = \int\limits_0^1 \dots \int\limits_0^1 \biggl(\int\limits_0^1 F(t_1, \dots, t_k)dt_j\biggr)^2
dt_1 \dots dt_{j - 1} dt_{j + 1} \dots dt_k.
\eeq

\smallskip
\noindent
{\bf Proposition 2'}.
{\it Let $\mathcal S_k$ denote the set of piecewise differentiable functions with the earlier given properties, including $I_k(F) \neq 0$ and $J_k^{(j)}(F) \neq 0$ for $1 \leqslant j \leqslant k$.
Let
\beq
\label{eq:2.11}
M_k = \sup \frac{\sum\limits_{j = 1}^k J_k^{(j)}(F)}{I_k(F)}, \ \ r_k = \left\lceil \frac{\theta M_k}{2}\right\rceil
\eeq
and let $\mathcal H$ be a fixed admissible sequence $\mathcal H = \{h_1, \dots, h_k\}$ of size~$k$.
Then there are infinitely many integers $n$ such that at least $r_k$ of the $n + h_i$ $(1 \leqslant i \leqslant k)$ are simultaneously primes.}

\smallskip
\noindent
{\bf Proposition 3'}.
{\it $M_{105} > 4$ and $M_k > \log k - 2 \log\log k - 2$ for $k > k_0$.}

\smallskip
\noindent
{\bf Remark}.
In the proof Maynard will use for every $k$ an explicitly given function $F = F_k$ satisfying the above inequality.
Therefore the additional dependence on $F$ will be actually a dependence on~$k$.

\smallskip
The main idea (beyond the original proof of Maynard--Tao) is that in the weighted sum $S_1$ in \eqref{eq:2.6} all those weights $w_n$ for numbers $n \in [N, 2N]$ are in total negligible for which any of the $n + h_i$ terms $(1 \leqslant i \leqslant k)$ has a small prime factor $p$ (i.e.\ with a sufficiently small $c_1(k)$ depending on $k$, $p \mid n + h_i$, $p < n^{c_1(k)}$).

To make it more precise let $c_1(k)$ be a sufficiently small fixed constant (to be determined later and fixed for the rest of the work).
Let $P^-(n)$ be the smallest prime factor of~$n$.
Then we have

\begin{lemma}
\label{lem:1}
We have
\beq
\label{eq:2.21}
S_1^- = \sum_{\substack{N \leqslant n < 2N\\ n \equiv \nu_0 \ (\text{\rm mod }W)\\
P^- \Bigl(\prod\limits_{i = 1}^k(n + h_i)\Bigr) < n^{c_1(k)}}}
w_n \ll_{k, H} \frac{c_1(k) \log N}{\log R} S_1.
\eeq
\end{lemma}

Since $R = N^{\frac{\theta}{2} - \varepsilon}$, $S_1^- / S_1$ will be arbitrarily small if $c_1(k)$ is chosen sufficiently small.
The proof of Lemma~\ref{lem:1} will be postponed to Section~\ref{sec:3}.
This means that during the whole proof we can neglect those numbers $n$ for which $P^- \biggl(\prod\limits_{i = 1}^k (n + h_i)\biggr) < n^{c_1(k)}$ and it is sufficient to deal with numbers $n$ with $n + h_i$ being almost primes for each $i = 1,2, \dots, k$ (by which we mean that $n + h_i$ has only prime factors at least $n^{c_1(k)}$).
A trivial consequence of this fact is that for such numbers $n$ \ \,$\prod\limits_{i = 1}^k(n + h_i)$ has a bounded number of prime factors.
Consequently we have for these numbers $n$ by (5.9) and (6.3)
\beq
\label{eq:2.22}
w_n \ll_{c_1(k), k} \lambda_{\max}^2 \ll_{c_1(k), k} y_{\max}^2 (\log R)^{2k} \ll_{c_1(k),k,F}(\log R)^{2k} \ll (\log R)^{2k}
\eeq
with the convention that the constants implied by the $\ll$ and $O$ constants can depend on $k$ and both $c_1(k)$ and $F = F_k$ will only depend on~$k$.

The essence of Maynard's proof is that (see (4.1)--(4.4) of \cite{May2015})
\beq
\label{eq:2.23}
S_2 > \left(\left(\frac{\theta}{2} - \varepsilon\right) (M_k - \varepsilon) + O\left(\frac1{D_0}\right)\right) S_1
\eeq
which directly implies the existence of infinitely many values $n$ such that there are at least
\beq
\label{eq:2.24}
r_k = \left\lceil \frac{\theta M_k}{2}\right\rceil
\eeq
primes among $n + h_i$ $(1 \leqslant i \leqslant k)$.

Let us denote, in analogy with \eqref{eq:2.6}
\beq
\label{eq:2.25}
S_1^+ := \sum_{\substack{n\\ N\leqslant n < 2N\\ n\equiv \nu_0 \ (\text{\rm mod }W) \\
P^-\bigl(\prod\limits_{i = 1}^k(n + h_i)\bigr) \geqslant n^{c_1(k)}}} \!\!\! w_n ,
\ \ \
S_2^+ := \sum_{\substack{n\\ N\leqslant n < 2N\\ n\equiv \nu_0 \ (\text{\rm mod }W) \\
P^-\bigl(\prod (n + h_i)\bigr) \geqslant n^{c_1(k)}}} \!\!\! w_n \biggl(\sum_{i = 1}^k \chi_{\mathcal P} (n + h_i)\biggr).
\eeq
Then Lemma~\ref{lem:1}, i.e.\ \eqref{eq:2.21} implies together with \eqref{eq:2.23} that (if $c_1(k)$ and $\varepsilon$ are chosen sufficiently small, $D_0$ sufficiently large, then)
\beq
\label{eq:2.26}
S_2^+ > \left(\left(\frac{\theta}{2} - \varepsilon\right) (M_k - \varepsilon) + O(c_1(k)) + O \left(\frac1{D_0}\right) + o(1)\right) S_1,
\eeq
which implies the existence of a large number of $n$ values in $[N, 2N)$, $n \equiv \nu_0$ $(\text{\rm mod }W)$ with at least $r_k$ primes among them and additionally almost primes with $P^-(n + h_i) > n^{c_1(k)}$ in all other components $i \in [1, k]$.

Together with \eqref{eq:2.22} this implies
\beq
\label{eq:2.27}
S_1^* := \sum_{\substack{n \\ N \leqslant n < 2N\\ n\equiv \nu_0 \ (\text{\rm mod }W)\\
P^-\bigl(\prod\limits_{i = 1}^k (n + h_i)\bigr) > n^{c_1(k)}\\
\#\{i; n + h_i \in \mathcal P\} \geqslant r_k}} \!\! 1
\gg \frac{S_1}{(\log R)^{2k}} = \frac{\left(1 + O\left(\frac1{D_0}\right)\right)\varphi(W)^k NI_k(F)}{W^{k + 1}(\log R)^k}.
\eeq

Since $D_0 = C^*(k)$ we have $\varphi(W)^k / W^{k + 1} \geqslant C'(k)$.
Thus a positive proportion (depending on $k$) of the integers $n \in [N, 2N)$ with $n \equiv \nu_0$ $(\text{\rm mod }W)$ and $P^-\biggl(\prod\limits_{i = 1}^k (n + h_i)\biggr) > n^{c_1(k)}$ contain at least $r_k$ primes among $n + h_i$ $(1 \leqslant  i \leqslant k)$.
This follows from \eqref{eq:2.27} and
\beq
\label{eq:2.28}
\sum_{\substack{N \leqslant n < 2N,\, n \equiv \nu_0 \ (\text{\rm mod }W)\\
P^-\bigl(\prod\limits_{i = 1}^k(n + h_i)\bigr) > n^{c_1(k)}}}\!\!\! 1 \ll \frac{N}{\log^k N}
\eeq
where the implied constant in the $\ll$ symbol depends only on $k$, $H$ and $c_1(k)$, therefore only on $k$, finally.
\eqref{eq:2.28} is a consequence of Selberg's sieve (see, for example,
Theorem 5.1 of \cite{HR1974} or Theorem~2 in \S~2.2.2 of \cite{Gre2001}).

If Lemma~\ref{lem:1} will be proved (see Section~\ref{sec:3}) then Theorem~\ref{th:2} will follow from Theorem~5 of \cite{Pin2010} which we quote here as

\smallskip
\noindent
{\bf Main Lemma}.
{\it Let $k$ be an arbitrary positive integer and $\mathcal H = \{h_1, \dots, h_k\}$ be an admissible $k$-tuple.
If the set $\mathcal N(\mathcal H)$ satisfies with constants $c_1(k)$, $c_2(k)$
\beq
\label{eq:2.29}
\mathcal N(\mathcal H) \subseteq \left\{n; P^-\biggl(\prod_{i = 1}^k (n + h_i)\biggr) \geqslant n^{c_1(k)}\right\}
\eeq
and
\beq
\label{eq:2.30}
\#\bigl\{n \leqslant X, n \in \mathcal N(\mathcal H)\bigr\} \geqslant \frac{c_2(k) X}{\log^k X}
\eeq
for $X > X_0$, then $N(\mathcal H)$ contains $\ell$-term arithmetic progressions for every~$\ell$.}

In order to see that the extra condition that the given prime pattern occurs also for consecutive primes we have to work in the following way.
For any given $\mathcal H = \{h_1, \dots, h_k\}$ with $k = \lceil Cm^2 \log m \rceil$ we choose an $m$-element subset $\mathcal H' = \{h_1', \dots, h_m'\} \subseteq \mathcal H$ with minimal diameter $h_m' - h_1'$
such that with some constants $c_1'(k), c_2'(k) > 0$ the relations \eqref{eq:2.29}--\eqref{eq:2.30}, more exactly
\beq
\label{eq:uj2.23}
\# \biggl\{n \leqslant X; \, P^{-} \biggl(\prod\limits_{i = 1}^k (n + h_i)\biggr) \geqslant n^{c_1'(k)}, \ n + h_i' \in \mathcal P \ (1 \leqslant i \leqslant m)\biggr\} \geqslant \frac{c_2'(k)X}{\log^k X}
\eeq
should hold for $X > X_0$.

By the condition that $\mathcal H'$ has minimal diameter we can delete from our set $\mathcal N(\mathcal H)$ those $n$'s for which there exists any $h_i \in \mathcal H \setminus \mathcal H'$, $h_1' < h_i < h_m'$
such that beyond \eqref{eq:uj2.23} also $n + h_i \in \mathcal P$ would hold.

On the other hand we can also neglect those $n \in \mathcal N(\mathcal H)$ for which with a given $h \in [1, H]$, $h \notin \mathcal H_k$ we would have additionally $n + h \in \mathcal P$ since the total number of such $h \in [1, H]$ is by \eqref{eq:2.28} at most
\beq
\label{eq:uj2.24}
O_k \left(\frac{NH}{\log^{k + 1} N}\right) = o\left(\frac{N}{\log^k N}\right)
\eeq
since our original $H$ in Theorem~\ref{th:2} was fixed.

We note that the above way of specifying the $m$-element sets $\mathcal H_m'$ for which we have arbitrarily long (finite) arithmetic progressions of $n$'s such that $n + h_i'$ $(1 \leqslant i \leqslant m)$ would be a given bounded pattern of \emph{consecutive} primes does not change the validity of the argument of Maynard (see Theorem 1.2 of \cite{May2015}) which shows that the above is true for a positive proportion of all $m$-element sets (the proportion depends on~$m$).

\section{Proof of Lemma~\ref{lem:1}. End of the proof of Theorem~\ref{th:2}}
\label{sec:3}

The proof of Lemma~\ref{lem:1} will be a trivial consequence of the following

\begin{lemma}
\label{lem:2}
The following relation holds for any prime $D_0 < p < N^{c_1}$ and all $i \in [1, \dots, k]$:
\beq
\label{eq:3.1}
S_{1, p}^* := \sum_{\substack{N \leqslant n < 2N\\
n \equiv \nu_0 \ (\text{\rm mod }W)\\
p \mid n + h_i}} w_n \ll_{F, H, k} \frac{\log p}{p \log R} \sum_{\substack{N \leqslant n < 2N\\
n \equiv \nu_0 \ (\text{\rm mod }W)}} w_n = \frac{\log p}{p \log R} S_1.
\eeq
\end{lemma}

\begin{proof}
It is clear that it is enough to show this for $i = 1$, for example.
During the proof we will use the analogue of Lemma 6 of \cite{GGPY2010} for the special case $k = 1$, $\delta = p \in \mathcal P$ and for squarefree $n$ with
\beq
\label{eq:3.2}
f(n) = n \ \ \ f_1(n) = \mu * f(n) = \prod_{p \mid n} (p - 1) = \varphi(n)
\eeq
which is as follows:
\beq
\label{eq:3.3}
T_p := \sum_{d, e} \frac{\lambda_d \lambda_e}{[d, e, p]/p} = \sum_{\substack{r\\ p + r}} \frac{\mu^2(r)}{\varphi(r)} (y_r - y_{rp})^2 .
\eeq
This form appears as the last displayed equation on page 85 of Selberg \cite{Sel1991} or equation (1.9) on page 287 of Greaves \cite{Gre2001}.
We note the general starting condition that similarly to \cite{May2015} the numbers $W$, $[d_1, e_1], \dots, [d_k, e_k]$ will be always coprime to each other.

Writing $n + h_1 = pm$ we see that we have for any $\varepsilon > 0$ and denoting $\sum^* $ for the conditions $n \in [N, 2N)$, $n \equiv \nu_0$ $(\text{\rm mod }W)$; $d_i, e_i \mod n + h_i$ $(2 \leqslant i \leqslant k)$
\beq
\label{eq:3.4}
S_{1,p} = \sum\nolimits_1 + \sum\nolimits_2 + O(R^{2 + \varepsilon})
\eeq
where
\beq
\label{eq:3.5}
\sum\nolimits_1 =  \underset{\substack{p\nmid [d_1, e_1]\\ d_1 \mid m, e_1 \mid m}}{\sum\nolimits^*} \lambda_{d_1, \dots, d_k} \lambda_{e_1, \dots, e_k},
\eeq
\beq
\label{eq:3.6}
\sum\nolimits_2 =  \underset{\substack{p\mid [d_1, e_1]\\ d_1 \mid pm, e_1 \mid pm}}{\sum\nolimits^*}
\lambda_{d_1, \dots, d_k} \lambda_{e_1, \dots, e_k}.
\eeq

Distinguishing further in $\sum_2$ according to $p^2 \mid d_1 e_1$ or not we obtain from \eqref{eq:3.5} and \eqref{eq:3.6} for any $\varepsilon > 0$
\beq
\label{eq:3.7}
\sum\nolimits_1 = \frac{N}{pW} \underset{p \nmid [d_1, e_1]}{\sum\nolimits^*} \frac{\lambda_{d_1, \dots, d_k} \lambda_{e_1, \dots, e_k}}{\prod\limits_{i = 1}^k [d_i, e_i]} + O (R^{2 + \varepsilon})
\eeq
and
\begin{align}
\label{eq:3.8}
\sum\nolimits_2
&= \frac{N}{pW} \Biggl\{\Biggl(\underset{d_1 = pd_1', p\nmid e_1}{\sum\nolimits^*} \frac{\lambda_{pd_1', \dots, d_k} \lambda_{e_1, \dots, e_k}}{[d_1', e_1] \prod\limits_{i = 2}^k [d_i, e_i]} +
\underset{e_1 = pe_1', p\nmid d_1}{\sum\nolimits^*}
\frac{\lambda_{d_1, \dots, d_k} \lambda_{pe_1', \dots, e_k}}{[d_1, e_1'] \prod\limits_{i = 2}^k [d_i, e_i]} \Biggr)
\\
&\quad + \underset{d_1 = pd_1',\, e_1 = pe_1'}{\sum\nolimits^*} \frac{\lambda_{pd_1', \dots, d_k} \lambda_{e_1'p, \dots, e_k}}{[d_1', e_1'] \prod\limits_{i = 1}^k [d_i, e_i]} \Biggr\} + O(R^{2 + \varepsilon}).
\nonumber
\end{align}
Consequently we have
\beq
\label{eq:3.9}
S_{1, p} = \frac{N}{pW} \sum\nolimits^* \frac{\lambda_{d_1,\dots, d_k} \lambda_{e_1, \dots, e_k}}{ \frac{[d_1, e_1, p]}{p} \prod\limits_{i = 2}^k [d_i e_i]} + O(R^{2 + \varepsilon}).
\eeq

Let us denote the sum in \eqref{eq:3.9} analogously to \eqref{eq:3.3} by $T_{p, 1}$.
Then, similarly to \eqref{eq:3.3} we obtain using additionally the argument of Section~5 of \cite{May2015}
\beq
\label{eq:3.10}
T_{p, 1} = \sum_{u_1, \dots, u_k} \frac{\prod\limits_{i = 1}^k \mu^2(u_i)}{\prod\limits_{i = 1}^k \varphi(u_i)} \bigl(y_{u_i, \dots, u_k} - y_{u_1p, u_2, \dots, u_k}\bigr)^2.
\eeq

However by the choice (6.3) of \cite{May2015} we have
\begin{align}
\label{eq:3.11}
\bigl(y_{u_1, \dots, u_k} - y_{u_1 p, u_2, \dots, u_k }\bigr)^2\!
&=\! F\! \left(\!\frac{\log u_1}{\log R}, ..., \frac{\log u_k}{\log R}\!\right)^{\!2} \!-\! F\!
\left(\! \frac{\log u_1 \! +\!  \log p}{\log R},..., \frac{\log u_k}{\log R}\! \right)^{\! 2}\\
&\ll_F \frac{\log p}{\log R} \ll \frac{\log p}{\log R},
\nonumber
\end{align}
since $F$ depends only on $k$, and hence the constant implied by the $\ll$ symbol may depend on $k$.
Hence we have by Proposition (4.1) of \cite{May2015}
\beq
\label{eq:3.12}
T_{p,1} \ll \frac{N}{W} \cdot \frac{\log p}{p \log R} \sum_{(u_1, W) = 1} \frac{\prod\limits_{i = 1}^k (\mu^2 (u_i))}{\prod\limits_{i = 1}^k \varphi(u_i)} \ll \frac{\log p}{p \log R} \cdot S_1
\eeq
which proves Lemma~\ref{lem:2} and thereby Lemma~\ref{lem:1} and Theorem~\ref{th:2}.
\end{proof}

\noindent
J\'anos Pintz\\
R\'enyi Mathematical Institute\\
of the Hungarian Academy of Sciences\\
Budapest, Re\'altanoda u. 13--15\\
H-1053 Hungary\\
e-mail: pintz.janos@renyi.mta.hu

\end{document}